\documentclass[a4paper,english,envcountsect]{svjour3}
\usepackage[T1]{fontenc}
\usepackage[utf8]{inputenc}
\usepackage{url}
\usepackage{amsmath}
\usepackage{amssymb}

\makeatletter

\pdfpageheight\paperheight
\pdfpagewidth\paperwidth

\numberwithin{equation}{section}
\numberwithin{figure}{section}

\usepackage{algorithm}
\usepackage[noend]{algpseudocode}
\usepackage[section]{placeins}

\usepackage[b5paper, 
            left=18mm, 
            right=18mm, 
            top=22mm, 
            bottom=22mm,
            footskip=10mm]{geometry}

\usepackage{microtype}
\makeatletter
\def\makeheadbox{\relax}
\makeatother

\usepackage[unicode, 
            colorlinks=true, 
            allcolors=blue, 
            breaklinks=true]{hyperref}

\makeatother

\usepackage{babel}

\begin{document}
\journalname{Journal of Optimization Theory and Applications}
\title{Quadratic Programming over Linearly Ordered Fields: Decidability and
Attainment of Optimal Solutions}
\titlerunning{Quadratic Programming over Linearly Ordered Fields}
\author{Dmytro O. Plutenko}
\authorrunning{D. O. Plutenko}
\institute{Dmytro O. Plutenko 
\at
Institute of Physics of the NAS of Ukraine \\
Prospekt Nauky 46, Kyiv, 03028, Ukraine \\
\email{dmytro.plutenko@gmail.com}}
\date{January 25, 2026}
\maketitle
\begin{abstract}
Classical existence theorems and solution methods for quadratic programming
traditionally rely on the analytical properties of real numbers, specifically
compactness and completeness. These tools are unavailable in general
linearly ordered fields, such as the field of rational numbers or
non-Archimedean structures, rendering standard analytical proofs insufficient
in these general algebraic settings. In this paper, we establish a
unified algebraic framework for the decidability of indefinite quadratic
programming subject to linear constraints over general linearly ordered
fields. We prove a generalized Eaves' theorem, demonstrating that
if a quadratic function --- encompassing convex, non-convex, or degenerate
(linear) cases --- is bounded from below on a polyhedron, the minimum
is attained within the field itself, regardless of topological completeness.
Our approach replaces classical analytical arguments with algebraic
induction on dimension and polyhedral decomposition. Based on this
foundation, we propose an exact, deterministic algorithm within the
Blum--Shub--Smale model of computation that decides boundedness
and computes a global minimizer using only field operations. We show
that the problem is solvable in finite time via a recursive search
over orthant-restricted facets. Finally, we note that linearly constrained
quadratic programming represents the maximal class of polynomial optimization
problems where exact solutions are structurally guaranteed within
the original field, thereby demarcating the algebraic boundary of
exact optimization over ordered structures.

\end{abstract}

\section{Introduction}

The origins of mathematical programming are deeply intertwined with
the dawn of cybernetics and the early development of computational
theory~\cite{Dantzig+1963}. Historically, the fundamental existence
theorems for both linear programming (LP)~\cite{GoldmanTucker}
and quadratic programming (QP) were established within the analytical
framework of the field of real numbers $\mathbb{R}$. The cornerstone
results of this era include the Frank--Wolfe theorem~\cite{Frank-Wolfe},
the direct existence proof by Blum and Oettli~\cite{BlumOettli1972},
and Eaves' landmark theorem~\cite{Eaves} establishing that any quadratic
function bounded from below on a polyhedron attains its minimum. These
classical proofs rely heavily on the topological properties of the
Euclidean space --- specifically, the compactness of closed bounded
sets and the Weierstrass extreme value theorem.

However, the formulation of a mathematical programming problem ---
defined by a system of linear inequalities and a polynomial objective
--- does not inherently require the Dedekind completeness or the
Archimedean property of the real numbers. A problem instance can be
posed over any linearly ordered field (LOF), such as the field of
rational numbers $\mathbb{Q}$ or the field of surreal numbers~\cite{conway2001numbers}.
In this broader algebraic setting, classical analytical proofs fail:
a bounded closed set in $\mathbb{Q}^{n}$ is not necessarily compact
due to the field's incompleteness. Consequently, the infimum of an
objective function may reside in the \textquotedbl gaps\textquotedbl{}
of the field, making the minimum unattainable within the original
structure.

The logical transition from $\mathbb{R}$ to more general fields was
partially addressed by the Tarski--Seidenberg transfer principle~\cite{Tarski1951}.
This principle ensures that first-order properties of $\mathbb{R}$
are preserved in all real closed fields, providing a theoretical foundation
for quantifier elimination algorithms~\cite{Basu2006Algorithms}.
In the context of optimization, quantifier elimination allows for
the automated decision of existence and the computation of optimal
values. However, real closed fields do not cover the entire spectrum
of LOFs. Specifically, the field $\mathbb{Q}$ lacks the intermediate
value property and the existence of roots for all non-negative elements,
placing it outside the scope of the Tarski--Seidenberg principle.

The mathematical landscape of LP has reached a state of foundational
maturity where its core theorems are not only generalized to LOFs
but have also been exhaustively validated through formal computer
verification. This state of completion, evidenced by rigorous treatments
in various proof assistants, confirms that the linear theory is fully
compatible with the abstract axiomatic framework of ordered structures~\cite{Allamigeon2019,LP-Duality-AFP}.
In stark contrast, QP has not undergone a similar systematic generalization.
While the linear case is considered a resolved frontier, the fundamental
possibility of extending QP results to the LOF setting remains largely
unexplored in existing literature. This work addresses this disparity,
focusing not merely on the computational procedures, but on proving
that the generalization of QP to linearly ordered fields is mathematically
viable, thereby extending the boundaries of axiomatic optimization
beyond the linear programming domain~\cite{dvorak2025}.

\noindent \textbf{Our Contribution.} The contribution of this paper
is focused on two primary objectives. First, we provide a theoretical
bridge between classical analysis and ordered algebra by establishing
a generalized Eaves' theorem for general linearly ordered fields.
By utilizing an induction scheme, we demonstrate that the existence
of an optimal solution is not a \textquotedbl topological accident\textquotedbl{}
of the real numbers but rather an intrinsic property of quadratic
forms on polyhedral sets.

Second, we we introduce a constructive recursive algorithm within
the Blum--Shub--Smale~\cite{BlumShubSmale1989} framework that
serves as a decision procedure for the three fundamental statuses
of a QP instance. This move from analytical tools to exact algebraic
procedures allows for a complete characterization of the problem without
the need for Dedekind completeness.

Finally, it is important to note that investigating the existence
and attainment of a minimum for higher-degree polynomial optimization
problems is generally futile within the LOF framework. Unlike the
linear and quadratic cases, where attainment is structurally guaranteed
even in the absence of completeness, the stationary conditions for
higher-order polynomials typically lead to non-linear equations whose
solutions may reside outside the original field. Consequently, QP
represents the maximal class of polynomial optimization problems where
exact solutions are structurally guaranteed within the underlying
field, effectively demarcating the algebraic boundary of exact optimization
over ordered structures.

\section{Preliminaries}

The primary goal of this section is to establish the geometric and
algebraic foundations required for our analysis. We operate in the
most general setting of optimization, avoiding assumptions of topological
completeness (such as the existence of limits or suprema) inherent
to real analysis. Instead, we rely strictly on the algebraic axioms
of a field and a total ordering.

To ensure clarity and consistency with established literature, we
adopt standard terminology from polyhedral geometry (see, e.g., Schrijver~\cite{Schrijver1986}).
However, we explicitly note that while classical texts often frame
these concepts within the context of real ($\mathbb{R}$) or rational
($\mathbb{Q}$) numbers, the structural results invoked here ---
specifically regarding implicit equalities, affine hulls, and dimension
reduction --- rely exclusively on finite algebraic operations (such
as Gaussian elimination). Consequently, these standard definitions
and properties remain fully valid for any linearly ordered field.

\subsection{Linearly Ordered Fields}

Throughout this paper, let $\mathbb{F}$ denote a\emph{ linearly ordered
field}. By definition, $\mathbb{F}$ is a field $(\mathbb{F},+,\cdot)$
equipped with a binary relation $\le$ that satisfies the axioms of
a \emph{total order} (reflexivity, antisymmetry, transitivity, and
totality) and is compatible with the field operations via the addition
and multiplication axioms~\cite{prestel1984lectures}. We emphasize
that we do not assume $\mathbb{F}$ to be a real closed field (i.e.,
polynomial equations may not have roots in $\mathbb{F}$), nor do
we assume it to be Archimedean. Our results thus hold for any structure
satisfying these algebraic axioms, including the rational numbers
$\mathbb{Q}$, the field of hyperreal numbers $\mathbb{R}^{*}$~\cite{NonStandardAnalysisRobinson},
or the field of surreal numbers~\cite{conway2001numbers}.

\subsection{\label{subsec:Polyhedra-and-implicit}Polyhedra and Implicit Equalities}

A \emph{polyhedron} $P\subseteq\mathbb{F}^{n}$ is defined as the
intersection of a finite number of closed half-spaces. Algebraically,
this corresponds to a system of $m$ linear inequalities:

\[
P=\{x\in\mathbb{F}^{n}:Ax\le b\},
\]

where $A\in\mathbb{F}^{m\times n}$ and $b\in\mathbb{F}^{m}$. In
accordance with the geometric definition of a half-space, we assume
that each row $a_{i}$ of the matrix $A$ is a non-zero vector, which
is a standard convention in polyhedral geometry (see, e.g., Schrijver~\cite[Sec.~7.2]{Schrijver1986}).

Following the structural decomposition described by Schrijver~\cite[Sec.~8.1]{Schrijver1986},
we classify the inequalities based on their behavior over $P$.

The system may contain inequalities that are satisfied as equalities
for every point in $P$. These are termed \emph{implicit equalities}.
Let $I^{=}$ denote the set of indices of such rows, and let $A^{=}x\le b^{=}$
represent the corresponding subsystem. The set of solutions to the
system of equations $A^{=}x=b^{=}$ defines the \emph{affine hull}
of $P$, denoted $\text{aff}(P)$.

The remaining inequalities, indexed by $I^{+}=\{1,\dots,m\}\setminus I^{=}$,
constitute the subsystem denoted by $A^{+}x\le b^{+}$. By definition,
for every inequality $a_{i}^{T}x\le b_{i}$ in this subsystem, there
exists at least one point $x\in P$ such that $a_{i}^{T}x<b_{i}$.

An inequality $a_{k}^{T}x\le b_{k}$ is called \emph{redundant} with
respect to the system if removing it does not alter the set of points
$P$. 

The \emph{dimension} of the polyhedron, $\dim(P)$, is defined as
the dimension of its affine hull. Since the affine hull is uniquely
determined by the system of implicit equalities $A^{=}x=b^{=}$ ---
which arises purely from the linear dependencies within the system
$Ax\le b$ --- the concept of dimension is well-defined and consistent
within the algebraic framework of any linearly ordered field $\mathbb{F}$.

\subsection{Full-Dimensionality and Interior Points}

To facilitate the analysis of quadratic functions on facets, it is
necessary to distinguish between polyhedra that have an interior and
those that are flattened into a lower-dimensional subspace.
\begin{definition}
[Full-Dimensionality]  A polyhedron $P\subseteq\mathbb{F}^{n}$ is
called \emph{full-dimensional} if $\dim(P)=n$. This is equivalent
to the condition that the system of implicit equalities is empty and
$P$ is not an empty set.
\end{definition}
Algebraically, full-dimensionality corresponds to the existence of
a \emph{strict interior point}. We say $x_{0}\in P$ is an interior
point if it satisfies all defining inequalities strictly: 
\[
Ax_{0}<b.
\]
 Conversely, if $P$ is not full-dimensional, no such point exists.
\begin{proposition}
[Dimension Reduction] Any non-empty polyhedron $P$ of dimension
$k<n$ is affinely isomorphic to a full-dimensional polyhedron $P'\subseteq\mathbb{F}^{k}$. 
\end{proposition}
\begin{proof}
This follows from the standard Gaussian elimination procedure over
$\mathbb{F}$. By expressing $n-k$ variables in terms of the $k$
free variables using the implicit equalities, we can project $P$
onto $\mathbb{F}^{k}$ using only field operations. \hspace*{\fill}\qed
\end{proof}

\subsection{Facets and Formal Facets}

We adopt the definition of a \emph{facet} as presented in the classic
work of Schrijver~\cite{Schrijver1986}. In particular, this definition
implies that if a polyhedron is described by a system of inequalities
containing no redundant constraints, then every linear constraint
(excluding implicit equalities) corresponds to exactly one facet.
In this case, the facet is simply the intersection of the polyhedron
with the hyperplane defined by the corresponding equality.

In contrast, we introduce the concept of a \emph{formal facet}, which
is defined relative to the specific algebraic description of the polyhedron
rather than its intrinsic geometry.
\begin{definition}
[Formal Facet] $P=\{x\in\mathbb{F}^{n}:Ax\le b\}$. For each row
index $i$, the set 
\[
F_{i}=P\cap\{x\in\mathbb{F}^{n}:a_{i}^{T}x=b_{i}\}
\]
 is defined as the \emph{formal facet} corresponding to the $i$-th
inequality.
\end{definition}
Thus, while a geometric facet is unique, a formal facet depends on
the representation: if the $i$-th inequality is redundant, the corresponding
formal facet might be empty or a face of lower dimension (or empty
set).

\section{Decoupled Diagonal Decomposition of Quadratic Functions\label{sec:Decoupled-diagonal-decomposition}}

In this section, we present a decoupled diagonal decomposition of
quadratic functions. We remark that while the optimization results
in subsequent sections rely on the ordering of $\mathbb{F}$, the
algebraic decomposition derived here requires only field operations.
Therefore, this lemma holds for \emph{any} field of characteristic
distinct from 2, as established by Lang~\cite{Lang2002Algebra},
regardless of whether the field is linearly ordered.
\begin{lemma}
[Decoupled Diagonal Decomposition]\label{lem:Decouple_Diagonal_Decomposition}
\end{lemma}
Let $f:\mathbb{F}^{n}\to\mathbb{F}$ be a quadratic function defined
by: 
\[
f(x)=x^{T}Qx+c^{T}x+\gamma,
\]
 where $Q\in\mathbb{F}^{n\times n}$ is symmetric, $c\in\mathbb{F}^{n}$,
and $\gamma\in\mathbb{F}$.

There exists a non-singular matrix $S\in\mathbb{F}^{n\times n}$,
a diagonal matrix $\Lambda\in\mathbb{F}^{n\times n}$, vectors $v,u\in\mathbb{F}^{n}$,
and a scalar $\gamma'\in\mathbb{F}$ such that: 
\begin{equation}
f(x)=(x-v)^{T}S^{T}\Lambda S(x-v)+u^{T}S(x-v)+\gamma',\label{eq:decoupled_form}
\end{equation}
 subject to the decoupling condition: 
\[
\Lambda u=\mathbf{0}.
\]
 This implies that the support of $u$ is disjoint from the support
of $\Lambda$ (i.e., $u_{i}\neq0$ only if $\Lambda_{ii}=0$ ). We
refer to Eq.~(\ref{eq:decoupled_form}) as the \emph{decoupled diagonal
form} of $f$.
\begin{proof}
First, we invoke the standard diagonalization result (e.g., \cite{Lang2002Algebra}).
There exists a non-singular matrix $S$ such that $Q=S^{T}\Lambda S$,
where $\Lambda=\text{diag}\left(\Lambda_{11},\Lambda_{22}\dots,\Lambda_{nn}\right)$.
Substituting $y=Sx$, the function becomes $f\left(y\right)=y^{T}\Lambda y+\tilde{c}^{T}y+\gamma$,
where $\tilde{c}=(S^{-1})^{T}c$. 

We define the shift vector $v'\in\mathbb{F}^{n}$ and the linear term
$u\in\mathbb{F}^{n}$ component-wise. For indices $i$ where $\Lambda_{ii}\neq0$,
we eliminate the linear term via completing the square. For indices
where $\Lambda_{ii}=0$, the linear term is retained: 
\[
\begin{cases}
v'_{i}=-\frac{\tilde{c}_{i}}{2\Lambda_{ii}},\quad u_{i}=0 & \text{if }\Lambda_{ii}\neq0,\\
v'_{i}=0,\quad u_{i}=\tilde{c}_{i} & \text{if }\Lambda_{ii}=0.
\end{cases}
\]
 The scalar $\gamma'$ is adjusted accordingly: 
\[
\gamma'=\gamma-\frac{1}{4}\sum_{i:\Lambda_{ii}\neq0}\Lambda_{ii}^{-1}\tilde{c}_{i}^{2}.
\]
 The function in $y$-coordinates takes the form: 
\[
f(y)=(y-v')^{T}\Lambda(y-v')+u^{T}y+\gamma'.
\]
 Note that by construction, $u^{T}v'=0$, so $u^{T}y=u^{T}(y-v')$.
Finally, setting $v=S^{-1}v'$ allows us to rewrite the expression
in terms of $x$, yielding Eq.~(\ref{eq:decoupled_form}). \hspace*{\fill}\qed
\end{proof}

\begin{corollary}
[Single-Direction Linear Term]\label{cor:Single_direcion_linear_term}
There exists a choice of the transformation matrix $S$ in Lemma~\ref{lem:Decouple_Diagonal_Decomposition}
such that the linear term vector $u$ has at most one non-zero component.
\end{corollary}
\begin{proof}
If $u=\mathbf{0}$, the condition holds trivially. If $u\neq\mathbf{0}$,
the condition $\Lambda u=\mathbf{0}$ implies that $u$ lies in the
kernel of the quadratic form. Let $K=\ker(\Lambda)$. We can perform
an additional linear transformation restricted to the subspace $K$
(which leaves the diagonal matrix $\Lambda$ invariant) to align the
first basis vector of $K$ with $u$. In this new basis, $u$ has
exactly one non-zero entry. \hspace*{\fill}\qed
\end{proof}

\begin{remark}
[Algorithmic Aspects] We highlight two properties relevant to the
algorithmic implementation:
\end{remark}
\begin{itemize}
\item \textbf{Constructivity:} The proof of Lemma~\ref{lem:Decouple_Diagonal_Decomposition}
is fully constructive. The matrix $S$ and $\Lambda$ are obtained
via a finite number of field operations, consistent with the classical
Lagrange's diagonalization algorithm~\cite[Ch.~X]{Gantmacher1959}.
Within the Blum--Shub--Smale model, such procedures are polynomial
in the dimension $n$~\cite{Blum1998}.
\item \textbf{Non-uniqueness of translation:} In the proof, the components
of the shift vector $v'$ corresponding to the kernel of $\Lambda$
can be chosen arbitrarily. Consequently, the resulting vector $v=S^{-1}v'$
is uniquely determined only on the range of $Q$. In our construction,
we set the kernel components of $v'$ to zero to obtain a solution
with minimal support, but any vector in the affine space $v+\ker(Q)$
is a valid center for the quadratic form.
\end{itemize}

\section{\label{sec:Existence-of-minima}Existence of Minima}

The existence of optimal solutions is a cornerstone of optimization
theory. For quadratic programming over real numbers ($\mathbb{R}$),
this property is well-established. Frank and Wolfe~\cite{Frank-Wolfe}
proved that a convex quadratic function bounded from below on a polyhedron
attains its minimum. Eaves~\cite{Eaves} generalized this result
to indefinite quadratic forms, demonstrating that if an indefinite
QP is bounded, an optimal solution exists.

However, Eaves' proof relies on the topological properties of $\mathbb{R}$---
specifically, the compactness of bounded closed sets (Heine--Borel
theorem) and the continuity of the objective function (Weierstrass
theorem). However, these properties do not hold in a general linearly
ordered field.

We now present an algebraic proof that extends Eaves' result to any
LOF, using induction on dimension instead of topological compactness.
\begin{theorem}
[Generalized Eaves' Theorem]\label{thm:Generalized_Eaves_Theorem}
\end{theorem}
Let $\mathbb{F}$ be a linearly ordered field. Let $f:\mathbb{F}^{n}\to\mathbb{F}$
be a quadratic function and $P\subseteq\mathbb{F}^{n}$ be a non-empty
polyhedron. If $f$ is bounded from below on $P$, then $f$ attains
its minimum on $P$.
\begin{proof}
We proceed by induction on the dimension of the space, $n$.

\textbf{Base case ($n=1$):} The polyhedron $P\subseteq\mathbb{F}^{1}$
is a connected subset of the line: a point, a closed interval, a ray,
or the whole line. The function is a univariate quadratic $f(x)=ax^{2}+bx+\gamma$. 
\begin{itemize}
\item If $a>0$ (strictly convex), the global minimum is at $x^{*}=-b/(2a)$.
If $x^{*}\in P$, the minimum is attained. If $x^{*}\notin P$, the
function is monotonic on $P$ and attains its minimum at the endpoint
closest to $x^{*}$.
\item If $a<0$ (concave) or $a=0,\ b\ne0$ (linear), the minimum must be
at an endpoint. If $P$ has no endpoint in the direction of descent,
$f$ would be unbounded below, contradicting the hypothesis. Thus,
the minimum is attained at a boundary point (vertex).
\item $a=0$, $b=0$ (trivial case), the minimum is attained at any point.
\end{itemize}
\textbf{Inductive hypothesis:} Assume the theorem holds for all quadratic
functions and polyhedra in spaces of dimension $n-1$ (and lower).

\textbf{Inductive step:} Consider a polyhedron $P\subseteq\mathbb{F}^{n}$
and a bounded quadratic function $f$.

\textbf{Case 1: Reduction to full dimensionality.} If $P$ is not
full-dimensional ($\dim(P)<n$), then $P$ lies entirely within a
proper affine subspace $H$ of dimension $k<n$. By identifying $H$
with $\mathbb{F}^{k}$, we can regard $f$ restricted to $P$ as a
problem in $\mathbb{F}^{k}$. Since $k\le n-1$, the inductive hypothesis
applies, and the minimum is attained.

\textbf{Case 2: Full-dimensional case.} Assume $\dim(P)=n$. We apply
the decoupled diagonal decomposition (Lemma~\ref{lem:Decouple_Diagonal_Decomposition})
and Corollary~\ref{cor:Single_direcion_linear_term} via the coordinate
transformation $y=S(x-v)$. Since this transformation is an affine
isomorphism, it preserves the combinatorics and dimension of the polyhedron.
For notational simplicity, we denote the polyhedron in the new coordinates
$y$ simply by $P$. The objective function becomes: 
\[
f(y)=y^{T}\Lambda y+u^{T}y+\gamma',
\]
 where $\Lambda$ is diagonal, $\Lambda u=\mathbf{0}$, and $u$ has
at most one non-zero component.

First, consider the boundary $\partial P$. It consists of a finite
union of facets $\{F_{j}\}$, each being a polyhedron of dimension
$n-1$. Since $f$ is bounded from below, it is bounded on each facet.
By the inductive hypothesis, $f$ attains a minimum on each $F_{j}$.
Let $a\in\partial P$ be the point achieving the global minimum over
the boundary: 
\[
f(a)=\min_{y\in\partial P}f(y).
\]
 Now we investigate whether there exists an interior point $y\in\text{int}(P)=P/\partial P$
such that $f(y)<f(a)$. We distinguish three sub-cases:

\textbf{Subcase 2.1: Convex form ($\Lambda\succeq0$ and $u=\mathbf{0}$).}\emph{
}The function is $f(y)=y^{T}\Lambda y+\gamma'$. This is a convex
function with the global unconstrained minimum at the origin $\mathbf{0}$. 

\textbf{Subcase 2.1.1:} If $\mathbf{0}\in P$, the global minimum
is attained at $\mathbf{0}$.

\textbf{Subcase 2.1.2:} If $\mathbf{0}\notin P$, consider any interior
point $y\in P$. The line segment connecting $y$ to the origin $\mathbf{0}$
must intersect the boundary $\partial P$ at some point $z$. Due
to the convexity of the form $y^{T}\Lambda y$, the function is non-decreasing
along the ray starting from $\mathbf{0}$. Since $z$ lies on the
segment between $\mathbf{0}$ and $y$, we have $f(z)\le f(y)$. But
$z\in\partial P$, so $f(a)\le f(z)$. Thus, $f(a)\le f(y)$ for all
interior points. The minimum is attained at $a$.

\textbf{Subcase 2.2: Linear descent ($u\neq\mathbf{0}$).} By Corollary~\ref{cor:Single_direcion_linear_term},
assume the linear term is aligned with the $k$-th coordinate, i.e.,
$u_{k}\neq0$ (which implies $\Lambda_{kk}=0$). Let $e_{k}\in\mathbb{F}^{n}$
denote the $k$-th standard basis vector (having $1$ at the $k$-th
position and $0$ elsewhere). The objective function depends linearly
on the coordinate $y_{k}$.

Consider an arbitrary interior point $y^{0}\in\text{int}(P)$. We
define the strict descent direction as $d=-\text{sign}(u_{k})e_{k}$.
Consider the ray emanating from $y^{0}$ in this direction: 
\[
R=\{y\in\mathbb{F}^{n}:y=y^{0}+td,\,t\ge0\}.
\]
 For any point on this ray, the function value is $f(y^{0}+td)=f(y^{0})-|u_{k}|t$.
Since $|u_{k}|>0$, the function strictly decreases as $t$ increases.
Because $f$ is bounded from below on $P$, the ray $R$ cannot be
entirely contained within $P$ indefinitely. Therefore, there exists
a scalar $t^{*}>0$ such that the point $z=y^{0}+t^{*}d$ lies on
the boundary $\partial P$. Since $t^{*}>0$, we have $f(z)<f(y^{0})$.
As $z\in\partial P$, it follows that $f(a)\le f(z)$. Thus, $f(a)<f(y^{0})$
for any interior point $y^{0}$, implying the minimum is attained
at $a$.

\textbf{Subcase 2.3: Concave direction ($u=\mathbf{0}$ and $\exists k:\Lambda_{kk}<0$).
}The objective function is strictly concave along the direction $e_{k}$.
For any interior point $y^{0}$, at least one of the rays in the directions
$e_{k}$ or $-e_{k}$ is strictly descending. Similarly to Case 2.2,
following this direction must eventually encounter the boundary $\partial P$
at a point $z$ (otherwise, $f$ would be unbounded below on $P$),
such that $f(z)<f(y^{0})$. It follows that $f(a)\leq f(z)<f(y^{0})$,
confirming that the minimum is attained on the boundary.

\textbf{Conclusion of Case 2:} The analysis of Subcases 2.1, 2.2,
and 2.3 is exhaustive. We have demonstrated that in every scenario,
the minimization problem either is solved by an interior point (specifically,
the unconstrained center $v$ in Subcase 2.1 when it belongs to $P$)
or reduces to a minimization problem on the boundary $\partial P$.
Since the boundary is a finite union of polyhedra of dimension $n-1$,
the inductive hypothesis guarantees the attainment of the minimum
on the boundary. Thus, in all cases, a global minimum exists on $P$.
The proof is complete. \hspace*{\fill}\qed
\end{proof}

The constructive nature of the proof of Theorem~\ref{thm:Generalized_Eaves_Theorem}
yields a precise characterization of the optimal solution's location.
We summarize these structural properties in the following corollary,
which serves as the theoretical basis for our exact algorithm.
\begin{corollary}
[Location of the Minimum]\label{cor:location-of-the-minimum} Let
$f(x)=x^{T}Qx+c^{T}x+\gamma$ be a quadratic function defined on a
full-dimensional polyhedron $P$. Adopting the notation of the decoupled
diagonal decomposition (Lemma~\ref{eq:decoupled_form}), the following
holds regarding the existence and location of the global minimum:
\begin{enumerate}
\item If $Q$ is positive definite ($\Lambda\succ0$) and $v\in P$, then
the global minimum exists, is unique, and is attained at the point
$v$.
\item If $Q$ is positive definite ($\Lambda\succ0$) and $v\notin P$,
then the global minimum exists, is unique, and is attained on the
boundary $\partial P$.
\item If $Q$ is positive semi-definite (\emph{$\Lambda\succeq0$}), $u=\mathbf{0}$,
and $v\in P$, then the global minimum exists, and $v$ is a global
minimizer (possibly non-unique).
\item If $Q$ is positive semi-definite (\emph{$\Lambda\succeq0$}), $u=\mathbf{0}$,
and $v\notin P$, then the global minimum exists and is attained on
the boundary $\partial P$.
\item If $u\neq\mathbf{0}$, then either $f$ is unbounded from below on
$P$, or a global minimum exists and is attained on the boundary $\partial P$.
\item If $u=\mathbf{0}$ and $\Lambda$ has at least one negative diagonal
entry, then either $f$ is unbounded from below on $P$, or a global
minimum exists and is attained on the boundary $\partial P$.
\end{enumerate}
\end{corollary}
\begin{proof}
Cases 1--4 concern convex functions bounded below by $\gamma'$ on
the entire space $\mathbb{F}^{n}$. Thus, boundedness on $P$ is automatic.
The location logic follows directly from Subcase 2.1 of Theorem~\ref{thm:Generalized_Eaves_Theorem}.
Specifically, strict convexity (Cases 1--2) implies strict inequality
$f(z)<f(y)$ along rays, ensuring uniqueness.

Cases 5 and 6 correspond to Subcases 2.2 and 2.3 of Theorem~\ref{thm:Generalized_Eaves_Theorem}.
In these scenarios, if the function is bounded from below, the proof
of Theorem~\ref{thm:Generalized_Eaves_Theorem} demonstrates that
for any interior point, there exists a descent direction leading to
the boundary. Thus, no interior point can be a global minimum. \hspace*{\fill}\qed
\end{proof}

\section{Algorithmic Decidability of Boundedness}

\subsection{Algorithmic Resolution and BSS Model}

We formalize our procedure within the framework of the Blum--Shub--Smale
model of computation \cite{BlumShubSmale1989}. In this model, the
arithmetic operations and comparisons over the field $\mathbb{F}$
are treated as primitive, unit-cost instructions. This abstraction
is particularly suitable for optimization over linearly ordered fields,
as it allows us to analyze the algorithm's complexity in terms of
the number of field operations, independent of the bit-length of the
coefficients. 

Our algorithm is deterministic and consists of a finite sequence of
field operations and branches based on the ordering of $\mathbb{F}$.
Crucially, the recursive structure ensures that for a fixed dimension
$n$, the total number of operations remains finite, thus satisfying
the termination criteria for a BSS machine over any ordered structure.

The generalized Eaves' theorem (Theorem~\ref{thm:Generalized_Eaves_Theorem})
establishes a fundamental result regarding the existence of solutions.
It implies that for any quadratic programming problem over a linearly
ordered field $\mathbb{F}$, exactly one of three mutually exclusive
scenarios holds:
\begin{enumerate}
\item The feasible set $P$ is empty (\emph{Infeasible}).
\item The quadratic function $f$ is unbounded from below on $P$ (\emph{Unbounded}).
\item The quadratic function $f$ attains a global minimum on $P$ (\emph{Optimal}).
\end{enumerate}

To make this theoretical classification constructive, we propose a
recursive procedure that explicitly determines the optimization status
and returns the minimizer if it exists. To formalize the procedure
within the Blum--Shub--Smale framework, we distinguish between the
abstract geometric objects and their algebraic encoding as data structures. 

We define a problem instance as a pair of structures $(\mathsf{P},\mathsf{f})$
over the field $\mathbb{F}$. Throughout Algorithm~\ref{alg:minqp_lof},
we use the sans-serif font to denote the data structures (the system
of constraints and the objective parameters), while the italic font
refers to the corresponding mathematical entities: 
\begin{itemize}
\item $\mathsf{P}=(A,b)$ represents the polyhedron $P=\{x\in\mathbb{F}^{n}:Ax\leq b\}$,
with $A\in\mathbb{F}^{m\times n},b\in\mathbb{F}^{m}$.
\item $\mathsf{f}=(Q,c,\gamma)$ represents the function $f(x)=x^{T}Qx+c^{T}x+\gamma$,
with $Q\in\mathbb{F}^{n\times n},c\in\mathbb{F}^{n},\gamma\in\mathbb{F}$.
\item $\mathsf{P}^{h}$ represents the sub-polyhedron $P^{h}$ restricted
to the orthant defined by $h\in\{-1,1\}^{n}$. 
\item As an exception, $\mathsf{P}'$ represents the instance associated
with the formal facet $F_{k}^{h}=P^{h}\cap\{y:a_{k}^{T}y=b_{k}\}$
(or its projection to $\mathbb{F}^{n-1}$). 
\item In the following description, $\mathsf{P}$ and $\mathsf{f}$ are
treated as dynamic objects subject to coordinate transformations and
dimensional reductions.
\end{itemize}
The algorithm \textsc{MinQP\_LOF} is designed to return a triplet
$(Status,Value,Point)$ which provides a complete characterization
of the instance: 
\begin{itemize}
\item Status\emph{=Infeasible} : Returned if $P=\emptyset$. In this case,
$Value=\infty$ and $Point=\text{null}$. 
\item Status\emph{=Unbounded}: Returned if $f$ is not bounded from below
on $P$. In this case, $Value=-\infty$ and $Point=\text{null}$.
\item Status\emph{=Optimal}: Returned if a global minimum exists. In this
case, $Value=f(x^{*})$ and $Point=x^{*}$, where$x^{*}\in P$ is
a global minimizer. 
\end{itemize}

\begin{algorithm}[htbp]
\caption{Exact Recursive Algorithm: \textsc{MinQP\_LOF}} \label{alg:minqp_lof}

\begin{algorithmic}[1]

\Require{$\mathsf{P} = (A, b)$ and  $\mathsf{f} = (Q, c, \gamma)$ over $\mathbb{F}$} 

\Ensure{Triplet $(\text{Status}, \text{Value}, \text{Point})$}

\Statex \textbf{Step 0: Preprocessing (zero-row elimination) } 
\For{each $i$-th row in $A$ where $a_i=\mathbf{0}$}
	\If{$b_i < 0$}
		\State \Return $(\text{"Infeasible"}, \infty, \text{null})$ 
	\Else 
		\State Remove constraint $i$ from $\mathsf{P}$
	\EndIf
\EndFor

\smallskip

\Statex \textbf{Step 1: Base case (dimension $n=1$)} 
\If{$n = 1$} 
    \State Solve univariate problem and 
	\Return $(\text{Status}, \text{Value}, \text{Point})$ 
\EndIf

\smallskip

\Statex \textbf{Step 2: Diagonalization of $f$} 
\State $(\mathsf{P}, \mathsf{f}) \gets \text{Apply coordinate transformation } y = S(x-v) $ \Comment{As per Lemma~\ref{lem:Decouple_Diagonal_Decomposition}}

\smallskip

\Statex \textbf{Step 3: Global optimality check}
\If{$\Lambda \succeq 0 \textbf{ and } u = \mathbf{0} \textbf{ and } \mathbf{0} \leq b$} 
\Comment{Check unconstrained minimizer}     
    \State \Return $(\text{"Optimal"}, \gamma', v)$ 
\EndIf

\smallskip

\Statex \textbf{Step 4: Branching over orthants and formal facets} 
\State $\mathcal{S} \gets \emptyset$ 
\Comment{Initialize set of minima candidates}
\For{each orthant defined by $h \in \{-1,1\}^n$}

    \State \ensuremath{\mathsf{P}^{h}\gets\left(\begin{bmatrix}A\\\text{diag}(h)\end{bmatrix},\begin{bmatrix}b\\\mathbf{0}\end{bmatrix}\right)} 
    \Comment{Sub-polyhedron $P^h$ restricted to the orthant}

    \For{each inequality with index $k$ in $\mathsf{P}^h$}

        \State Compute $M \in \mathbb{F}^{n \times (n-1)}$ and $p \in \mathbb{F}^n$ such that           
       \Statex \hskip \algorithmicindent \hskip \algorithmicindent $y = M z + p$ parameterizes the hyperplane $a_k^T y = b_k$

		\State $(\mathsf{P}', \mathsf{f}') \gets \text{Substitute } y = M z + p \text{ into } \mathsf{f} \text{ and the remaining constraints of } \mathsf{P}^h$

        \Comment{Project formal facet $F_k^h = P^h \cap \{y : a_k^T y = b_k\}$ to $\mathbb{F}^{n-1}$}

        \State $(stat, val, z^*) \gets \Call{MinQP\_LOF}{\mathsf{P}', \mathsf{f}'}$ 
        \Comment{Solve sub-problem on the $F_k^h$}
        \If{$stat = \text{"Unbounded"}$}
            \Return $(\text{"Unbounded"}, -\infty, \text{null})$ 
        \EndIf
        \If{$stat = \text{"Optimal"}$} $\mathcal{S} \gets \mathcal{S} \cup \{(val, M z^* + p)\}$ 
        \EndIf
    \EndFor
\EndFor

\smallskip

\Statex \textbf{Step 5: Final selection} 
\If{$\mathcal{S} = \emptyset$} 
		\Return $(\text{"Infeasible"}, \infty, \text{null})$ 
\EndIf 

\State Find $(val^*, y^*) \in \mathcal{S}$ with the minimum $val^*$ \Comment{Pick global minimum}
\State \Return $(\text{"Optimal"}, val^*, S^{-1}y^*+v)$

\end{algorithmic} 
\end{algorithm}

\begin{theorem}
[Correctness and Termination]\label{algo_correctness} Let $\mathbb{F}$
be a linearly ordered field. For any input $(\mathsf{P},\mathsf{f})$
defined over $\mathbb{F}$, Algorithm~\ref{alg:minqp_lof} terminates
in a finite number of steps and returns: 
\end{theorem}
\begin{enumerate}
\item \emph{Infeasible} if and only if $P=\emptyset$. 
\item \emph{Unbounded} if and only if the function $f$ is unbounded from
below on $P$. 
\item \emph{Optimal} and a point $x^{*}\in P$ if and only if a global minimum
exists, in which case $f(x^{*})\le f(x)$ for all $x\in P$. 
\end{enumerate}
\begin{remark}
[Data Structures vs. Geometric Objects]\label{rem:data_vs_geom}
It is important to distinguish between the input data structure $\mathsf{P}=(A,b)$,
which may contain trivial constraints with $a_{i}=\mathbf{0}$, and
the geometric polyhedron $P$ it represents. While our definition
of a polyhedron in Section~\ref{subsec:Polyhedra-and-implicit} requires
a standard system of non-zero rows, the algorithm is designed to accept
any $\mathsf{P}$ as valid input. 

Specifically, to ensure consistency, we map any input instance containing
conflicting rows ($a_{i}=\mathbf{0}$ and $b_{i}<0$) to the empty
polyhedron $P=\emptyset$. If no such conflicting rows exist, then
after removing all trivially satisfied rows ($a_{i}=\mathbf{0}$ and
$b_{i}\ge0$), the resulting system $(A',b')$ consists strictly of
non-zero rows for the matrix $A'$. This reduced form conforms exactly
to the definition of a polyhedron $P$ established in Section~\ref{subsec:Polyhedra-and-implicit}.
Step~0 serves as the bridge between these representations, ensuring
that all subsequent geometric analyses are performed on the correctly
reduced standard system.
\end{remark}
\FloatBarrier

\subsection{Supporting Lemmas}

The correctness of the recursive logic in Algorithm~\ref{alg:minqp_lof}
relies on several geometric properties of polyhedra. Specifically,
we must justify why the optimization status of a problem instance
can be reliably determined by examining only its lower-dimensional
facets (after partitioning it by orthants). This section establishes
the lemmas required to validate this reduction.
\begin{lemma}
\label{lem:slicing} Let $P\subseteq\mathbb{F}^{n}$ be a polyhedron
and $s\in\mathbb{F}^{n}$ be a non-zero vector. For $\alpha\in\mathbb{F}$,
define the slice $P_{\langle s,\alpha\rangle}=P\cap\{x:s^{T}x=\alpha\}$.
A function $f:\mathbb{F}^{n}\to\mathbb{F}$ is bounded from below
on $P$ if and only if there exists a uniform constant $\epsilon\in\mathbb{F}$
such that for all $\alpha$ where $P_{\langle s,\alpha\rangle}\neq\emptyset$,
the inequality $f(x)\ge\epsilon$ holds for all $x\in P_{\langle s,\alpha\rangle}$. 

\end{lemma}
\begin{proof}
\textbf{Necessity ($\implies$):} Suppose $f$ is bounded below on
$P$, i.e., there exists $\nu\in\mathbb{F}$ such that $f(x)\ge\nu$
for all $x\in P$. We simply choose $\epsilon=\nu$. Since any non-empty
slice $P_{\langle s,\alpha\rangle}$ is a subset of $P$, the inequality
$f(x)\ge\epsilon$ holds for all points in the slice.

\textbf{Sufficiency ($\impliedby$):} We proceed by contradiction.
Assume the condition holds (there exists a uniform $\epsilon$ such
that $f(x)\ge\epsilon$ on every non-empty slice), but $f$ is \emph{not}
bounded from below on $P$. By definition of unboundedness, there
exists a point $y\in P$ such that $f(y)<\epsilon$. Let $\alpha=s^{T}y$.
By construction, $y\in P_{\langle s,\alpha\rangle}$. However, according
to the hypothesis, the inequality $f(z)\ge\epsilon$ must hold for
all $z\in P_{\langle s,\alpha\rangle}$. Specifically, $f(y)\ge\epsilon$.
We have arrived at the contradiction $f(y)<\epsilon$ and $f(y)\ge\epsilon$.
Thus, our assumption was false, and $f$ is bounded from below on
$P$. \hspace*{\fill}\qed
\end{proof}

The finite termination of Algorithm~\ref{alg:minqp_lof} is guaranteed
by its recursive structure; a rigorous proof of this property is provided
within the proof of Theorem~\ref{algo_correctness}. Upon termination,
the algorithm returns one of three possible optimization statuses.
In this section, we first analyze the \emph{Infeasible} status and
establish the following lemma to prove that the algorithm correctly
identifies the emptiness of the feasible region.
\begin{lemma}
[Correctness of Infeasibility Detection]\label{lem:infeasibility_proof}
Algorithm~\ref{alg:minqp_lof} returns \emph{Infeasible} if and only
if the polyhedron $P$ represented by input $\mathsf{P}$ is empty.
\end{lemma}
\begin{proof}
We prove the statement by induction on the dimension $n$ of the problem.

\textbf{1. Base case ($n=1$):} The algorithm solves the univariate
problem in Step 1. Feasibility is determined by a direct check of
the interval constraints (Step 0 and Step 1). The output is correct
by inspection of the univariate cases.

\textbf{2. Inductive hypothesis:} Assume the algorithm works correctly
for dimension $n-1$.

\textbf{3. Inductive step (dimension $n$):} 

The proof relies on two fundamental properties of polyhedra over ordered
fields: 
\begin{enumerate}
\item \emph{Orthant decomposition:} $P$ is non-empty if and only if its
intersection with at least one orthant is non-empty. This follows
from the fact that the union of all $2^{n}$ closed orthants covers
$\mathbb{F}^{n}$. 
\item \emph{Boundary property:} A non-empty polyhedron $P\subseteq\mathbb{F}^{n}$
defined by $m$ closed half-spaces either coincides with the entire
space ($m=0$) or has at least one non-empty formal facet $F_{k}=P\cap\{y:a_{k}^{T}y=b_{k}\}$.
\[
P\neq\emptyset\iff(m=0)\lor(\exists k\in[m]:F_{k}\neq\emptyset).
\]
\end{enumerate}
\textbf{Necessity ($\implies$):} If the algorithm returns \emph{Infeasible},
it implies that one of the following scenarios occurred: 
\begin{enumerate}
\item Step~0 identified a trivial contradiction where $a_{i}=\mathbf{0}$
and $b_{i}<0$. According to the Remark \ref{rem:data_vs_geom} this
directly implies $P=\emptyset$.
\item Step~0 passed without identifying trivial contradictions, but the
algorithm proceeded through the subsequent steps as follows:
\begin{itemize}
\item Step~3 confirmed that the origin in the $y$-basis is not a feasible
minimizer ($\mathbf{0}\notin P$ or the optimality conditions were
not met).
\item In Step~4, every recursive call on every formal facet $F_{k}^{h}$
of every orthant returned \emph{Infeasible}. Crucially, no recursive
call returned \emph{Unbounded}, as that would have triggered an immediate
exit with the \emph{Unbounded} status.
\item Consequently, the set of candidates $\mathcal{S}$ remained empty
at the end of Step~4.
\item In Step~5, the condition $\mathcal{S}=\emptyset$ was satisfied,
and the algorithm returned \emph{Infeasible}.
\end{itemize}
\end{enumerate}
By the inductive hypothesis, the fact that all recursive calls returned
\emph{Infeasible} means that all formal facets $F_{k}^{h}$ are empty
sets. According to the boundary property and orthant decomposition,
the emptiness of all facets implies that $P=\emptyset$.

\textbf{Sufficiency ($\impliedby$):} Conversely, if $P=\emptyset$,
then all sub-problems on facets are also defined on empty sets. By
the inductive hypothesis, they return \emph{Infeasible}, leading to
$\mathcal{S}=\emptyset$ and the correct output. \hspace*{\fill}\qed
\end{proof}

\begin{lemma}
[Boundedness Propagation on Orthants]\label{lem:boundness-propagation-orthants}Let
$f(y)=y^{T}\Lambda y+u^{T}y+\gamma$ be a quadratic function in decoupled
diagonal form (where $\Lambda$ is diagonal, $\Lambda u=\mathbf{0}$,
and $u$ has at most one non-zero component). Let $P\subseteq\mathbb{F}^{n}$
be a polyhedron restricted to an orthant (i.e., defined by a system
$Ay\le b$ which includes constraints of the form $y_{i}\le0$, $i\in\left[n\right]$).
If the dimension $n\ge2$, then $f$ is bounded from below on $P$
if and only if $f$ is bounded from below on every formal facet of
$P$.
\end{lemma}
\begin{proof}
\textbf{Necessity ($\implies$):} is trivial: if $f$ is bounded on
$P$, it is bounded on any subset of $P$.

\textbf{Sufficiency ($\impliedby$): }The cases where $\Lambda\succeq0$
and $u=\mathbf{0}$ are trivial, as the function is inherently bounded.
Similarly, if $P$ is empty or a lower-dimensional polyhedron (where
the entire set coincides with its formal facets), the statement holds
by definition. We thus focus on the case where $P$ is full-dimensional
and either $\Lambda$ has at least one negative entry or $\Lambda\succeq0$
with $u\neq\mathbf{0}$.

We rely on Lemma~\ref{lem:slicing}, which states that if there exists
a vector $s\in\mathbb{F}^{n}$ such that $f$ restricted to every
slice $P_{\langle s,\alpha\rangle}=P\cap\{y:s^{T}y=\alpha\}$ is bounded
from below by a certain value for all $\alpha$, then $f$ is bounded
on $P$. Let $\epsilon_{1}$ be the lower bound of $f$ on all formal
facets of $P$.

Our goal is to find a vector $s$ such that the function on each slice
is bounded. To ensure that every non-empty slice $P_{\langle s,\alpha\rangle}$
is a bounded polyhedron (polytope), we restrict our search to vectors
$s$ with strictly negative components ($s_{i}<0$ for all $i=1,\dots,n$).
Given that $P$ is contained in the negative orthant ($y_{i}\le0$),
this choice of $s$ implies that $y_{i}\in[s_{i}^{-1}\alpha,0]$,
and consequently, $\alpha$ must be non-negative for the slice to
be non-empty.

To analyze $f$ on a slice, we perform variable elimination. Let $y_{\{j\}}$
denote the vector $y$ with the $j$-th component removed, and $\Lambda_{\{jj\}}$
denote the matrix $\Lambda$ with the $j$-th row and column removed.
From $s^{T}y=\alpha$, we have $y_{j}=s_{j}^{-1}(\alpha-s_{\{j\}}^{T}y_{\{j\}})$.
Substituting this into $f$, we obtain: 
\begin{align*}
\tilde{f}(y_{\{j\}}) & =y_{\{j\}}^{T}\left[\Lambda_{\{jj\}}+\Lambda_{jj}s_{j}^{-2}s_{\{j\}}s_{\{j\}}^{T}\right]y_{\{j\}}\\
 & +\left[u_{\{j\}}^{T}-2\alpha s_{j}^{-2}\Lambda_{jj}s_{\{j\}}^{T}-u_{j}s_{j}^{-1}s_{\{j\}}^{T}\right]y_{\{j\}}\\
 & +\left[\alpha^{2}s_{j}^{-2}\Lambda_{jj}+\alpha u_{j}s_{j}^{-1}+\gamma\right].
\end{align*}

We select the index $j$ and the vector $s$ according to the following
cases:

\textbf{Case 1:}\emph{ $\Lambda$ has at least one negative diagonal
entry.} Let $j$ be an index such that $\Lambda_{jj}<0$. By choosing
$s_{j}$ sufficiently small in absolute value, the matrix $[\Lambda_{\{jj\}}+\Lambda_{jj}s_{j}^{-2}s_{\{j\}}s_{\{j\}}^{T}]$
becomes concave along $s_{\{j\}}$. According to Corollary~\ref{cor:location-of-the-minimum},
the minimum of $\tilde{f}$ on the bounded polytope $P_{\langle s,\alpha\rangle}$
is attained on its boundary. Since the boundary of the slice lies
on the formal facets of $P$, it follows that the inequality $f\left(x\right)\ge\epsilon_{1}$
holds for all $x$ in the slice.

\textbf{Case 2:}\emph{ $\Lambda\succeq0$, $\Lambda$ contains at
least two zero diagonal elements, and $u\neq\mathbf{0}$.} Since $u$
has exactly one non-zero component, we can choose an index $j$ such
that $\Lambda_{jj}=0$ and $u_{j}=0$. After elimination, the reduced
function $\tilde{f}$ on the slice retains a linear term $u_{k}y_{k}$
for some index $k\neq j$ where $\Lambda_{kk}=0$ and $u_{k}\neq0$.
In this direction, the function is purely linear. According to Corollary~\ref{cor:location-of-the-minimum}
(part~5), since the function contains a linear component with a zero
quadratic coefficient, the minimum on the bounded polytope $P_{\langle s,\alpha\rangle}$
must be attained on the boundary. Therefore the inequality $f\left(x\right)\ge\epsilon_{1}$
holds for all $x$ in the slice.

\textbf{Case 3:}\emph{ $\Lambda\succeq0$, only one diagonal element
$\Lambda_{jj}=0$, and $u\neq\mathbf{0}$.} In this case, necessarily
$u_{j}\neq0$. The function on the slice becomes: 
\[
\tilde{f}=y_{\{j\}}^{T}\Lambda_{\{jj\}}y_{\{j\}}-u_{j}s_{j}^{-1}s_{\{j\}}^{T}y_{\{j\}}+\left[\alpha u_{j}s_{j}^{-1}+\gamma\right].
\]
 \textbf{Subcase 3.1:} If $u_{j}s_{j}^{-1}<0$, the linear term pushes
the minimizer to the boundary of the slice according to Corollary~\ref{cor:location-of-the-minimum}
(part~2). Therefore the inequality $f\left(x\right)\ge\epsilon_{1}$
holds for all $x$ in the slice.

\textbf{Subcase 3.2:} If $u_{j}s_{j}^{-1}\ge0$, then for any $\alpha\ge0$,
the term $\alpha u_{j}s_{j}^{-1}+\gamma$ is bounded from below. Since
$\Lambda_{\{jj\}}\succ0$ (as the only zero was removed), the quadratic
part is bounded below. Consequently, the function $f$ is bounded
from below by $\gamma$ on each such slice. \hspace*{\fill}\qed
\end{proof}

The establishing of boundedness on an orthant-restricted polyhedron
leads to the following result regarding the location of the optimal
solution. Under the conditions of Lemma~\ref{lem:boundness-propagation-orthants},
the search for a global minimum can be restricted to the boundary
of $P$.
\begin{corollary}
[Location of the Minimum on Orthants] \label{cor:global_min_facets}Let
$f$ and $P$ ($n\ge2$) be defined as in Lemma~\ref{lem:boundness-propagation-orthants}.
If $P$ is non-empty and $f$ is bounded from below on $P$, then
the global minimum of $f$ on $P$ is attained on at least one of
its facets $F_{k}$: 
\[
\min_{y\in P}f(y)=\min_{k}\left(\min_{y\in F_{k}}f(y)\right).
\]
 
\end{corollary}
\begin{proof}
By Corollary~$\ref{cor:location-of-the-minimum}$, the set $\{0\}\cup\partial P$
is guaranteed to contain a global minimizer $y^{*}$. Due to the orthant
constraints, the origin is not an interior point of $P$. Consequently,
$y^{*}$ must lie on at least one facet. \hspace*{\fill}\qed
\end{proof}

\subsection{Proof of Theorem~\ref{algo_correctness}}

We are now ready to establish the overall correctness of the proposed
method. The following proof synthesizes the geometric properties of
orthant-restricted polyhedra and the inductive structure of the algorithm's
recursive calls.
\begin{proof}
The termination of the algorithm is guaranteed by the fact that in
each recursive call (Step 4), the dimension $n$ of the problem is
strictly reduced by one. Since $n$ is finite and the number of branching
orthants ($2^{n}$) and formal facets ($m+n$) is also finite at each
level, the procedure results in a finite recursion tree.

According to Theorem~\ref{thm:Generalized_Eaves_Theorem}, exactly
one of the three statuses ---\textbf{ }\emph{Infeasible}, \emph{Unbounded},
or \emph{Optimal} --- must characterize any instance over $\mathbb{F}$.
Given that Lemma~\ref{lem:infeasibility_proof} establishes the correctness
of the \emph{Infeasible} detection, it remains to prove the following
two forward implications by induction on $n$.

\textbf{The base case $n=1$:} (Step~1) is handled by a direct algebraic
solution of the univariate quadratic problem over the feasible interval,
which correctly identifies the status and the minimizer.

\textbf{Correctness of }\textbf{\emph{Unbounded}}\textbf{ output:}
Assume the hypothesis holds for $n-1$. The algorithm returns \emph{Unbounded}
only in Step 4, when a recursive call on a formal facet $F_{k}^{h}$
returns \emph{Unbounded}. By the inductive hypothesis, $f$ is truly
unbounded from below on the subset $F_{k}^{h}\subseteq P$. Since
$f$ is unbounded on a subset of $P$, it is necessarily unbounded
on $P$.

\textbf{Correctness of }\textbf{\emph{Optimal}}\textbf{ output: }Assume
the hypothesis holds for $n-1$. The algorithm returns \emph{Optimal}
either in Step 3 (where the unconstrained stationary point is feasible)
or in Step 5. In Step 5, the status is returned only if $P\neq\emptyset$
and no formal facet was found to be unbounded. For each orthant $h$,
the boundedness of all formal facets implies the boundedness of $P^{h}$
(Lemma~\ref{lem:boundness-propagation-orthants}). Crucially, Corollary~\ref{cor:global_min_facets}
(location of the minimum on orthants) guarantees that the global minimum
of an orthant-restricted polyhedron must be attained on at least one
of its facets. Since the algorithm exhaustively collects and compares
optimal points from all formal facets into the set $\mathcal{S}$,
the selection of the smallest value in Step 5 correctly identifies
the exact global minimum of $f$ on $P$.

The proof is concluded by observing that the three statuses --- \emph{Infeasible},
\emph{Unbounded}, and \emph{Optimal} --- form a partition of all
possible states for a quadratic program over $\mathbb{F}$, as guaranteed
by Theorem~\ref{thm:Generalized_Eaves_Theorem}. Since the results
established above demonstrate that if the algorithm returns a specific
status, the problem instance necessarily possesses the corresponding
property, the mutual exclusivity of these statuses implies the converse.
Specifically, for any problem instance, Algorithm~\ref{alg:minqp_lof}
identifies its unique true optimization status and, in the case of
optimality, provides a certificate in the form of a global minimizer
$x^{*}$. The proof is complete. \hspace*{\fill}\qed
\end{proof}

\begin{remark}
[Complexity Note]The recursive structure of Algorithm~\ref{alg:minqp_lof}
mirrors the inductive proof of Theorem~\ref{thm:Generalized_Eaves_Theorem}.
Since the recursion depth is bounded by $n$ and the number of branches
is finite, termination is guaranteed. However, the branching over
orthants and facets at each step leads to a super-exponential worst-case
time complexity in terms of the dimension $n$. This behavior is consistent
with the NP-hard nature of the indefinite quadratic programming problem
and reflects the cost of an exact structural decomposition without
numerical approximations.
\end{remark}

\section{Results and Discussion}

Theorem~\ref{thm:Generalized_Eaves_Theorem} establishes a comprehensive
generalization of the classical existence theory by extending Eaves'
theorem~\cite{Eaves} to the framework of arbitrary linearly ordered
fields. This result demonstrates that the attainment of a minimum
for a bounded quadratic function is an intrinsic algebraic characteristic
of the problem's linear-quadratic structure. Such a finding clarifies
that the existence of an optimal solution does not rely on the continuity
or compactness inherent to $\mathbb{R}$, which were central to classical
analytical proofs.

In particular, a consequence of this theorem is the exhaustive classification
of any quadratic programming problem into three mutually exclusive
statuses: \emph{Infeasible} (polyhedron $P$ is an empty set), \emph{Unbounded}
(function $f$ is unbounded on $P$), or \emph{Optimal} (existence
of the global minimizer). The theorem restricts the possible outcomes
to these three cases, ensuring that if a feasible set is non-empty
and the function is bounded from below, an optimal solution must exist
within the original field $\mathbb{F}$.

The constructive nature of the proof of Theorem~\ref{thm:Generalized_Eaves_Theorem}
yields a precise characterization of the optimal solution's location,
serving as the theoretical basis for our exact algorithm. This mechanism
is rooted in the decoupled diagonal decomposition introduced in Section~\ref{sec:Decoupled-diagonal-decomposition}.
As established in Corollary~\ref{cor:location-of-the-minimum}, the
coefficients of the diagonal matrix $\Lambda$ and the linear vector
$u$ act as structural markers of the possible problem's status and
the position of the minimizer. For instance, if the decomposition
reveals a non-zero linear term $u_{k}\neq0$ for a free variable,
or if $\Lambda$ contains at least one negative diagonal element,
the problem is identified as either \emph{Unbounded} or as having
a minimum that necessarily lies on the boundary.

From an algorithmic perspective, we have proven that this classification
is decidable within the Blum--Shub--Smale model of computation.
The developed algorithm, \texttt{MinQP\_LOF}, determines the status
of the problem in a finite number of steps and, in the \emph{Optimal}
case, computes the exact coordinates of the global minimizer. Although
the computational complexity is inherently high due to the recursive
nature of the polyhedral search and the necessity of orthant decomposition,
the algorithm is guaranteed to terminate with a correct result over
any LOF.

These findings encompass the most general form of QP, including degenerate
cases of linear programming. While for LP the attainment of a minimum
over LOFs is a known result, our work provides a necessary and rigorous
extension to the indefinite quadratic case.

The boundaries of this algebraic approach become evident when considering
higher-degree polynomials or alternative algebraic structures. For
polynomial optimization of degree three or higher, the structural
existence theorem generally fails. Consider the cubic function $f(x)=x^{3}-6x$
on the interval $[0,3]\cap\mathbb{Q}$; its minimum lies at $x=\sqrt{2}$,
which is absent in $\mathbb{Q}$, hence it is not attained despite
being bounded. Similarly, the transition from fields to linearly ordered
rings disrupts attainment. In the ring $R=\mathbb{Z}[1/3]=\{a/3^{k}:a\in\mathbb{Z},k\in\mathbb{N}\}$,
which consists of integers enriched with fractions where the denominators
are powers of three, the function $f(x)=(2x-1)^{2}$ is bounded below,
yet its infimum $0$ is not attained as $1/2\notin R$. These observations
suggest that the algebraic structure of a field provides a critical
foundation for ensuring solution attainment, whereas simpler ordered
structures like rings or higher-degree models may lack the requisite
properties even for well-posed problems.

\section{Conclusions}

In this paper, we have established the decidability and attainment
of optimal solutions for quadratic programming over general linearly
ordered fields. By proving a generalized version of Eaves' theorem,
we demonstrated that the existence of a minimum is an inherent algebraic
property of quadratic forms, independent of topological completeness.
We introduced the exact recursive algorithm \texttt{MinQP\_LOF},
providing a constructive solution within the Blum--Shub--Smale model.
These findings effectively delineate the boundaries of exact solvability,
confirming that quadratic programming represents the maximal class
of polynomial optimization problems where solution attainment is structurally
guaranteed across the entire spectrum of linearly ordered fields.

\begin{small}
\section*{Acknowledgments} The author is grateful to Prof. Mykhail V. Vasnetsov for his interest in this work, his encouragement, and for facilitating the communication of these results.

\section*{Declarations}

\textbf{Funding:} This research was supported by the Institute of Physics of the National Academy of Sciences of Ukraine.

\noindent\textbf{Conflict of interest:} The author declares that he has no conflict of interest.

\noindent\textbf{Data availability:} Data sharing is not applicable to this article as no datasets were generated or analyzed during the current study.

\end{small}
\clearpage

\sloppy
\bibliographystyle{spmpsci}
\bibliography{bib}
\end{document}